\documentclass[10pt,a4paper]{amsart}

\usepackage[english]{babel}
\usepackage[T1]{fontenc}
\usepackage{palatino}
\usepackage{verbatim}
\usepackage{amsmath}
\usepackage{amssymb}
\usepackage{amsthm}
\usepackage{soul}
\usepackage{amsfonts}
\usepackage{graphicx}
\usepackage{mathtools}
\usepackage{hyperref}
\usepackage{comment}
\usepackage{hyperref}
\usepackage{bbm}
\usepackage[usenames,dvipsnames]{xcolor}
\usepackage{stackrel}

\newcommand{\bbH}{\mathbb{H}}
\newcommand{\bbL}{\mathbb{L}}
\newcommand{\N}{\mathbb{N}}

\newcommand{\R}{\mathbb{R}}
\newcommand{\bbS}{\mathbb{S}}

\newcommand{\calH}{{\mathcal H}}

\newcommand{\calL}{\mathcal{L}}
\newcommand{\calN}{\mathcal{N}}
\newcommand{\calV}{\mathcal{V}}

\newcommand{\bfX}{\mathbf{X}}

\newcommand{\g}{\mathfrak{g}}

\newcommand{\1}{\mathbbm{1}}

\newcommand{\Int}{\operatorname{Int}}

\newcommand{\norm}[1]{\Vert#1\Vert}
\newcommand{\Per}{\operatorname{Per}}

\newcommand{\spt}{\operatorname{spt}}

\newcommand{\scal}[1]{\langle#1\rangle}

\numberwithin{equation}{section}

\theoremstyle{plain}
\newtheorem{theorem}{Theorem}[section]

\newtheorem{lemma}[theorem]{Lemma}
\newtheorem{corollary}[theorem]{Corollary}

\theoremstyle{definition}
\newtheorem{definition}[theorem]{Definition}

\theoremstyle{remark}
\newtheorem{remark}[theorem]{Remark}

\title[Monotone sets and local minimizers for the perimeter]{Monotone sets and local minimizers for the perimeter in Carnot groups}

\author{S\'everine Rigot}
\address[Rigot]{Universit\'e C\^ote d'Azur, CNRS, LJAD, France}
\email{Severine.RIGOT@univ-cotedazur.fr}

\begin{document}

\begin{abstract}  Monotone sets have been introduced about ten years ago by Cheeger and Kleiner who reduced the proof of the non biLipschitz embeddability of the Heisenberg group into $L^1$ to the classification of its monotone subsets. Later on, monotone sets played an important role in several works related to geometric measure theory issues in the Heisenberg setting. In this paper, we work in an arbitrary Carnot group and show that its monotone subsets are sets with locally finite perimeter that are local minimizers for the perimeter. Under an additional condition on the ambient Carnot group, we prove that their measure-theoretic interior and support are precisely monotone. We also prove topological  and measure-theoretic properties of local minimizers for the perimeter whose interest is independent from the study of monotone sets. As a combination of our results, we get in particular a sufficient condition under which any monotone set admits measure-theoretic representatives that are precisely monotone.
 \end{abstract}

\maketitle

\section{Introduction}

Monotone sets have been introduced by Cheeger and Kleiner in \cite{MR2729270} where the proof of the non biLipschitz embeddability of the Heisenberg group into $L^1$ is reduced to the classification of its monotone subsets, see also \cite{MR2892612}. Later on, this classification played an important role in several works related to the study of locally finite perimeter sets and to geometric measure theory issues in the Heisenberg setting \cite{MR3815462,MR4127898,https://doi.org/10.48550/arxiv.2004.12522,https://doi.org/10.48550/arxiv.2105.08890}.

\smallskip

In this paper we will work in an arbitrary Carnot group. Our results concern both monotone sets and local minimizers for the perimeter. We will first prove that monotone sets are sets with locally finite perimeter that are local minimizers for the perimeter, Theorems~\ref{thm:monotone-implies-locally-finite-perimeter} and~\ref{thm:monotone-minimizer-for-perimeter}. Next, we prove topological and measure-theoretic properties of local minimizers for the perimeter, Theorem~\ref{thm:minimizer-for-perimeter}, whose interest is independent from the study of monotone sets. We will finally go back to monotone sets. We prove that under an additional assumption on the ambient Carnot group both the measure-theoretic interior and support of a monotone set are precisely monotone, see~(H) and Theorem~\ref{thm:int-spt-precisely-monotone}. As a consequence of our results, we get in particular that under condition~(H) any monotone set admits precisely monotone measure-theoretic representatives, Corollary~\ref{cor:equivalent-int-spt-precisely-monotone}, which was the initial motivation for the present paper. We will come back to this initial motivation later on in this introduction.

\smallskip

To describe our results, let us consider a Carnot group $G$. We denote by $\oplus_{i=1}^s \g_i$ a stratification of its Lie algebra of left-invariant vector fields, see Section~\ref{sect:preliminaries} for precise definitions, and by $\mu$ a Haar measure on $G$.  A line in $G$ is the unparameterized oriented image of an integral curve of a left-invariant vector field in $\g_1 \setminus \{0\}$. According to \cite[Definition~3.7]{MR2729270} a subset $E$ of $G$ is \textit{monotone} if it is $\mu$-measurable and for a.e. line $L$ the restriction of its characteristic function to $L$ agrees a.e. with a monotone function, equivalently, $L\cap E$ coincides up to a null subset of $L$ with a connected subset of $L$ that has connected complement in $L$, see Section~\ref{subsect:monotone-sets}.

\smallskip

We will first prove that monotone sets have locally finite perimeter, meaning that for any $X\in\g_1$ the distributional derivate of their characteristic function along $X$ is a Radon measure.

\begin{theorem} \label{thm:monotone-implies-locally-finite-perimeter}
Monotone subsets of $G$ have locally finite perimeter.
\end{theorem}

We stress that monotonicity combined with properties of sets with locally finite perimeter have already been used in the literature. However, monotone sets that were considered were known, for some other reason than monotonicity, or assumed, to have locally finite perimeter. The fact that  monotonicity on its own implies locally finite perimeter was not explicitely known before, at least to our knowledge, and will be obtained as a rather direct consequence of the definitions.

\smallskip

We will next prove that monotone sets are local minimizers for the perimeter.  We refer to Section~\ref{sect:monotone-local-min-perimeter} for precise definitions, see in particular Definition~\ref{def:local-minimizer-perimeter}.

\begin{theorem} \label{thm:monotone-minimizer-for-perimeter}
Given any scalar product $\scal{\cdot,\cdot}$ on $\g_1$ with associated collection $\bfX$ of orthonormal bases of $(\g_1,\scal{\cdot,\cdot})$, monotone subsets of $G$ are local minimizers for the $\bfX$-perimeter.
\end{theorem}

Theorem~\ref{thm:monotone-minimizer-for-perimeter} will be obtained as a consequence of the kinematic formula \cite[Proposition~3.13]{MR2165404} which relates the $\bfX$-perimeter to 1-dimensional perimeter on lines. It should be compared with  \cite[Proposition~3.9]{https://doi.org/10.48550/arxiv.2105.08890} where local minimality of locally finite perimeter monotone subsets of the Heisenberg group is shown inside open convex sets. Note that $G$ is here an arbitrary Carnot group and local minimality will be proved in any open set. Recall also, as already mentioned, that thanks to Theorem~\ref{thm:monotone-implies-locally-finite-perimeter} we can drop the assumption about local finiteness of the perimeter which appears in~\cite[Proposition~3.9]{https://doi.org/10.48550/arxiv.2105.08890}.

\smallskip

We will then turn to the study of local minimizers for the $\bfX$-perimeter. Such local minimality conditions have already been widely studied, especially in connection with the Bernstein problem, see for instance \cite{MR2333095,MR3406514,MR3984100,MR4069613}. We prove topological and measure-theoretic properties that nevertheless do not seem to be available in the literature and have their own interest. They will be obtained by rather standard arguments. Theorem~\ref{thm:minimizer-for-perimeter} should in particular be compared with \cite[Section~5]{MR2000099} where global minimizers for the perimeter under a volume constraint are considered. 

\smallskip

Given a subset $E$ of $G$ we denote by $\Int(E)$, $\overline{E}$, and $\partial E$ its topological interior, closure, and boundary. Given a Borel measure $\nu$ on $G$ we denote by $\spt \nu$ its support. Given a $\mu$-measurable subset $E$ of $G$ we denote by $\spt_\mu (E)$ its measure-theoretic support, i.e., the support of the restriction of $\mu$ to $E$. We denote by $\Int_\mu(E) = G \setminus \spt_\mu (E^c)$ its measure-theoretic interior and by  $\partial_\mu E = \spt_\mu (E) \cap \spt_\mu (E^c)$ its measure-theoretic boundary. Here $E^c = G \setminus E$ denotes the complement of $E$. We denote by $E_1$ and $E_0$ the set of $\mu$-density 1, respectively $\mu$-density 0, points for $E$ with respect to some, equivalently any, left-invariant homogeneous distance on $G$, see Section~\ref{subsect:distances}, and by $\partial^* E = G\setminus (E_0 \cup E_1)$ the essential boundary of $E$. Given a distance $d$ on $G$, we denote by $B_d(x,r) = \{y\in G:\, d(x,y) < r\}$ the open ball with center $x\in G$ and radius $r>0$ with respect to the distance $d$.

\begin{theorem} \label{thm:minimizer-for-perimeter}
Given any scalar product $\scal{\cdot,\cdot}$ on $\g_1$ with associated collection $\bfX$ of orthonormal bases of $(\g_1,\scal{\cdot,\cdot})$ and given any left-invariant homogeneous distance $d$ on $G$, there is $\epsilon>0$ such that for any local minimizer $E$ for the $\bfX$-perimeter,
\begin{equation}
\partial_\mu E = \left\{x\in G:\, \min\left\{\frac{\mu(E\cap B_d(x,r))}{\mu(B_d(x,r))}, \frac{\mu(E^c \cap B_d(x,r))}{\mu(B_d(x,r))}\right\} >\epsilon \, \,\forall r>0\right\}, \label{e:meas-theoretic-boundary-uniform-lower-densities}
\end{equation}
and consequently,
\begin{align}
& \Int_\mu(E) = E_1 \text{ and }  \spt_\mu (E) = G \setminus E_0, \label{e:interior-and-support} \\
& \spt \Per_\bfX(E,\cdot) = \partial_\mu E = \partial^* E =\partial \Int_\mu(E) = \partial \spt_\mu(E), \label{e:boundaries}\\
& \overline{\Int_\mu(E)} = \spt_\mu(E) \text{ and } \Int(\spt_\mu(E)) = \Int_\mu(E). \label{e:regular-open-closed}
\end{align}
\end{theorem}

We refer to Section~\ref{sect:monotone-local-min-perimeter} for more comments that put the conclusions of Theorem~\ref{thm:minimizer-for-perimeter} into perspective, see also Remark~\ref{rmk:AR-CB}. Note  that, as an immediate consequence of \eqref{e:interior-and-support} and \eqref{e:boundaries}, we get the following corollary.

\begin{corollary} \label{cor:minimizer-for-perimeter} Given any scalar product $\scal{\cdot,\cdot}$ on $\g_1$ with associated collection $\bfX$ of orthonormal bases of $(\g_1,\scal{\cdot,\cdot})$ and given any local minimizer $E$ for the $\bfX$-perimeter, both $\Int_\mu(E)$ and $\spt_\mu(E)$ are measure-theoretic representatives of $E$, and given any left-invariant homogeneous distance $d$ on $G$, the restriction of $\calH^{Q-1}_d$ to $\partial_\mu E$ is a locally finite measure.
\end{corollary}

Here, given $E,F \subset G$, we say that $F$ is a measure-theoretic representative of $E$ if $F$ is equivalent to $E$, i.e., $\mu(E\triangle F) = 0$ where $E\triangle F = (E\setminus F) \cup (F\setminus E)$ denotes the symmetric difference between $E$ and $F$. We denote by $Q=\sum_{i=1}^s i\dim \g_i$ the homogeneous dimension of $G$ and by $\calH^{Q-1}_d$ the $(Q-1)$-dimensional Hausdorff measure with respect to the distance $d$.

\smallskip

We will finally go back to monotone sets and investigate precise monotonicity of their measure-theoretic interior and support. We say that a subset $E$ of $G$ is \textit{precisely monotone} if for every line $L$ both $E\cap L$ and $L\setminus E$ are connected.  As already mentioned, our initial motivation for the present work was the question to know whether monotone sets admit measure-theoretic representatives that are precisely monotone. This is in turn motivated by attempts to classify monotone sets, see \cite{https://doi.org/10.48550/arxiv.2106.13490} and Remark~\ref{rmk:classification} for more comments about classification issues. On the one hand, given a $\mu$-measurable set $E$, the sets $E_1$ and $G\setminus E_0$ are natural, and well known, measure-theoretic representatives of $E$. When $G$ is abelian, i.e. $s=1$, then $E_1$ and $G\setminus E_0$ can be proved to be precisely monotone when $E$ is monotone \cite[Lemma~64]{MR3815462}. However, any reasonnable attempt to extend in a direct way this property in non abelian Carnot groups should involve both left- and right-translations and seems to be hardly reachable because right-translations are not Lipschitz with respect to left-invariant homogeneous distances, and vice-versa. On the other hand, considering the measure-theoretic interior and support turns out to be more natural when dealing with monotonicity, see \cite[Section~5]{MR2729270} and  Theorem~\ref{thm:int-spt-precisely-monotone}. However, neither $\Int_\mu(E)$ nor $\spt_\mu(E)$ are measure-theoretic representatives of $E$ in general. Therefore, using Theorem~\ref{thm:monotone-minimizer-for-perimeter} about local minimality of monotone sets and Corollary~\ref{cor:minimizer-for-perimeter} about local minimizers for the perimeter gives a way to overcome this issue. 

\smallskip

Going back to precise monotonicity, we give in the next theorem a sufficient condition on the ambient Carnot group under which for any monotone subset $E$ both $\Int_\mu(E)$ and $\spt_\mu(E)$ are precisely monotone.

\begin{theorem} \label{thm:int-spt-precisely-monotone}
Assume that 
\begin{itemize}
\item[(H)] for any $X\in \g_1 \setminus \{0\}$ there is an integer $p\geq 2$ such that the map $\Gamma_p:(\g_1)^p \rightarrow G$ defined by $\Gamma_p(X_1,\dots,X_p) = \exp X_1 \cdots \exp X_p$ is open at $(X,\dots,X) \in (\g_1)^p$.
\end{itemize}
Then for any monotone subset $E$ of $G$ both $\Int_\mu(E)$ and $\spt_\mu(E)$ are precisely monotone.
\end{theorem}

As explained above, we then get from  Theorem~\ref{thm:monotone-minimizer-for-perimeter}, Corollary~\ref{cor:minimizer-for-perimeter}, and Theorem~\ref{thm:int-spt-precisely-monotone} the following corollary.

\begin{corollary} \label{cor:equivalent-int-spt-precisely-monotone}
Assume that (H) holds. Then for any monotone subset $E$ of $G$ both $\Int_\mu(E)$ and $\spt_\mu(E)$ are precisely monotone measure-theoretic representatives of $E$. 
\end{corollary}

Theorem~\ref{thm:int-spt-precisely-monotone} is inspired by \cite[Proposition~5.8]{MR2729270} where the ambient space is the Heisenberg group and where for $p=2$ the  map $\Gamma_2$ is open (it is actually a submersion) at any $(X,Y)\in \g_1 \times \g_1$ such that $X+Y\not=0$. Note that condition~(H) does not require the map $\Gamma_p$ to be a submersion but only to be open. Although there are examples of Carnot groups where condition~(H) does not hold \cite[Section~4]{MR4208090}, it is known to hold in  step-2 Carnot groups~\cite[Theorem~2.1]{MR4119259}. Therefore, Corollary~\ref{cor:equivalent-int-spt-precisely-monotone} gives in particular a positive answer to the existence of precisely monotone measure-theoretic representatives for monotone sets in the step-2 setting.

\smallskip

The rest of this paper is organized as follows. Section~\ref{subsect:monotone-sets} is devoted to the definition of monotone sets. We will in particular give explicit representations of 1-dimensional measures on lines, as well as of measures on the set of lines, that will be particularly helpful in the rest of the paper. We recall some well known facts about left-invariant homogeneous distances that will be relevant for our purposes in Section~\ref{subsect:distances}. Theorems~\ref{thm:monotone-implies-locally-finite-perimeter} and~\ref{thm:monotone-minimizer-for-perimeter} are proved in the first part of Section~\ref{sect:monotone-local-min-perimeter} while Theorem~\ref{thm:minimizer-for-perimeter} is proved in the second part of the section. We conclude this paper with the proof of Theorem~\ref{thm:int-spt-precisely-monotone} in Section~\ref{sect:int-spt-precisely-monotone}.

\smallskip

\noindent \textit{Acknowledgement.} We would like to thank Daniele Morbidelli for several useful discussions.

\section{Preliminaries} \label{sect:preliminaries}
In the rest of this paper, we consider a Carnot group $G$, i.e., a connected, simply connected Lie group whose Lie algebra $\g$ of left-invariant vector fields admits a stratification. We let $(\g_1,\dots,\g_s)$ denote a stratification of $\g$, i.e., $s$ is a positive integer, $\g_1,\dots,\g_s$ are linear subspaces of $\g$ such that $\g = \g_1 \oplus \cdots \oplus \g_s$ and $[\g_1,\g_i]  =\g_{i+1}$ for $i=1,\dots,s$, where $[\g_1,\g_i]$ denotes the linear subspace of $\g$ generated by $[X,Y]$ with $X\in\g_1$, $Y\in \g_i$, $\g_s \not=\{0\}$ and $\g_i = \{0\}$ for $i>s$. We recall that any two stratifications of a stratifiable Lie algebra are isomorphic (see~\cite[Proposition~2.17]{MR3742567} for a precise statement) and stress that all notions considered in the present paper are invariant under isomorphisms of stratified Lie algebras. We let $\exp: \g \mapsto G$ denote the exponential map and recall that it is a smooth diffeomorphism from $\g$ to $G$. We let $\mu$ denote a Haar measure on $G$. We refer to \cite{MR3587666} and the references therein for a comprehensive introduction to Carnot groups. 

\subsection{Monotone sets}  \label{subsect:monotone-sets}

Given a left-invariant vector field $X\in\g$, we denote by $\Phi_X:G\times \R\rightarrow G$ its associated flow, namely,
\begin{equation*}
\Phi_X(x,t) = x\cdot \exp(tX).
\end{equation*}

In the rest of this paper, we equip $\R$ with its usual 1-dimensional Lebesgue measure denoted by either $\calL^1$ or $dt$. We stress that all Borel measures considered in the present paper are tacitly extended as outer measures, i.e., non-negative countably sub-additive set functions defined on the whole power set of the underlying space, see \cite[Chapter~1~(1.2)]{MR1333890}.

\smallskip

A line in $G$ is the unparameterized oriented image of an integral curve of a left-invariant vector field in $\g_1^* = \g_1 \setminus \{0\}$. Given a line $L$, we write $L=(X,x)\in \g_1^* \times G$ to mean that $L=\Phi_X(x,\R)$ and $X$ is the orientation of $L$. Lines are homeomorphic to $\R$ and we say that a subset  $A$ of a line $L$ is a null subset of $L$ if $\Phi_* \calL^1(A)=0$ for some, equivalently any, homeomorphism $\Phi$ from $\R$ to $L$, where $\Phi_* \calL^1$ denotes the image of the measure $\calL^1$ on $\R$ under the homeomorphism $\Phi$. We say that two subsets of a line $L$ are equivalent if their symmetric difference within $L$ is a null subset of $L$. 

\smallskip

Given $E\subset G$, we denote by $\1_E$ its characteristic function.

\begin{definition} \label{def:monotone-line}
We say that a line $L$ is monotone for a subset $E$ of $G$ if for some, equivalently any, homeomorphism $\Phi$ from $\R$ to $L$ the function $t\in \R \mapsto \1_E(\Phi(t))$ agrees $\calL^1$-a.e. with a monotone function, equivalently, if there is a connected subset of $L$ that is equivalent to $E\cap L$ and has connected complement in $L$.
\end{definition}

The collection of all lines in $G$ has a natural smooth structure that induces a natural notion of null sets of lines. Namely, let $\scal{\cdot, \cdot}$ be a scalar product on $\g_1$. Let $\norm\cdot$ denote its associated norm and set $\bbS(\g_1) = \{X\in\g_1:\, \norm X =1\}$. Given $ X \in \bbS(\g_1)$, we set $L_{X}=\exp \R X$. We let $X^\perp$ denote the direct sum of the orthogonal complement of $\R X$ in $(\g_1,\scal{\cdot, \cdot})$ with $\oplus_{i= 2}^s \g_i$ and set $N_{X} = \exp X^\perp$. It is well known that for any $X \in \bbS(\g_1)$ both $L_{X}$ and $N_{X}$ are subgroups of $G$ that are complementary, i.e., $N_{X} \cap L_{X} = \emptyset$ and $N_{X} \cdot L_{X} = G$. There is a one to one correspondence between lines in $G$ and elements in 
\begin{equation*} 
\bbL_{\scal{\cdot, \cdot}}(G) = \bigsqcup_{X \in \bbS(\g_1)} \{X \} \times N_{X}.
\end{equation*}
Indeed, since $N_{X}$ and $L_{X}$ are complementary subgroups, it can easily be seen that $L$ is a line in $G$ if and only if there is $(X ,x) \in \bbL_{\scal{\cdot, \cdot}}(G)$ such that $L=(X,x)$ and such an $(X ,x) \in \bbL_{\scal{\cdot, \cdot}}(G)$ is unique. We can therefore identify the collection of all lines in $G$ with $\bbL_{\scal{\cdot, \cdot}}(G)$ equipped with its natural structure of smooth orientable manifold. Then we say that a set of lines $A$ is a null set if $\calN(A)=0$ for some, equivalently any, smooth measure with positive density $\calN$ on $\bbL_{\scal{\cdot, \cdot}}(G)$. We stress that any choice of scalar product on $\g_1$ leads to the same null sets of lines.

\begin{definition} \cite[Definition~3.7]{MR2729270} \label{def:monotone-sets} We say that $E$ is monotone if $E$ is a $\mu$-measurable subset of $G$ and almost every line in $G$ is monotone for $E$.
\end{definition}

Note that $E$ is monotone if and only if its complement $E^c = G \setminus E$ is monotone.

\begin{remark} We indeed have that a $\mu$-measurable subset of $G$ is monotone in the sense of Definition~\ref{def:monotone-sets} if and only if it is monotone in the sense of \cite[Definition~3.7]{MR2729270}. Working with oriented lines, rather than unoriented ones as is \cite{MR2729270}, is  nothing more than  a technical convenience that was also adopted in subsequent works involving monotone subsets of the Heisenberg group (see for instance \cite{MR2892612}).
\end{remark}

\begin{remark} Monotonicity is preserved under left-translations, meaning that if $E$ is monotone then $y\cdot E$ is monotone for any $y\in G$. Indeed, since left-translations are homeomorphisms that send lines to lines, it is clear from Definition~\ref{def:monotone-line} that if a line $L$ is monotone for a set $E$ then the line $y\cdot L$ is monotone for $y\cdot E$ for any $y\in G$. Next, if $A$ is a null set of lines then $y\cdot A$ is a null set of lines for any $y\in G$. Indeed, let $y\in G$ and a scalar product $\scal{\cdot,\cdot}$ on $\g_1$ be given. For any $L=(X,x)\in \bbL_{\scal{\cdot, \cdot}}(G)$, $t\in \R$, we have
\begin{equation*}
y\cdot \Phi_X(x,t) =  y_{N_X} \cdot \exp(t_y X) \cdot x \cdot \exp(-t_y X) \cdot \exp((t_y + t)X).
\end{equation*}
Here $(y_{N_X},t_y)$ is the unique element in $N_X \times \R$ such that $y= y_{N_X} \cdot \exp(t_y X)$. Since $N_X$ is a normal subgroup of $G$, it follows that, as an element in $\bbL_{\scal{\cdot, \cdot}}(G)$, the left-translate $y\cdot L$ of a line $L=(X,x) \in \bbL_{\scal{\cdot, \cdot}}(G)$ is given by $y\cdot L = (X,\Psi_{X,y}(x))\in \bbL_{\scal{\cdot, \cdot}}(G)$ where $\Psi_{X,y}(x) =y \cdot  x \cdot \exp(-t_y X)$. Since the map $(X,x) \in \bbL_{\scal{\cdot, \cdot}}(G) \mapsto (X,\Psi_{X,y}(x)) \in  \bbL_{\scal{\cdot, \cdot}}(G)$ is a smooth diffeomorphism, it sends null sets of lines to null sets of lines.
\end{remark}

\begin{remark} \label{rmk:monotonicity-dilations} For similar reasons, monotonicity is preserved under dilations, meaning that if $E$ is monotone then $\delta_\lambda(E)$ is monotone for any $\lambda>0$. Here $\delta_\lambda:G \rightarrow G$ denotes the dilation with factor $\lambda$ given by $\delta_\lambda (\exp \sum_{i=1}^s X_i) = \exp \sum_{i=1}^s \lambda^i X_i$ for any $X_i\in \g_i$, $i=1,\dots,s$. Indeed, since dilations are group homomorphisms, we have $\delta_\lambda(\Phi_X(x,\R)) = \Phi_X(\delta_\lambda(x),\R)$ for any $(X,x) \in \g_1^* \times G$, $\lambda>0$. Therefore dilations are homeomorphisms that send lines to lines and it follows from Definition~\ref{def:monotone-line} that if a line $L$ is monotone for a set $E$ then the line $\delta_\lambda(L)$ is monotone for $\delta_\lambda(E)$ for any $\lambda>0$. Next, if $A$ is a null set of lines then $\delta_\lambda(A)$ is a null set of lines for any $\lambda >0$.  Indeed, fix a scalar product $\scal{\cdot, \cdot}$ on $\g_1$. Since $N_X$ is a homogeneous subgroup of $G$ for any $X\in \bbS(\g_1)$, meaning that $\delta_\lambda(N_X) \subset N_X$ for any $\lambda>0$, it follows that, as an element in $\bbL_{\scal{\cdot, \cdot}}(G)$, the dilation $\delta_\lambda(L)$ of a line $L=(X,x) \in \bbL_{\scal{\cdot, \cdot}}(G)$ is given by $\delta_\lambda(L) = (X,\delta_\lambda(x))\in \bbL_{\scal{\cdot, \cdot}}(G)$ for any $\lambda>0$. Since for $\lambda>0$ the map $(X,x) \in \bbL_{\scal{\cdot, \cdot}}(G) \mapsto (X,\delta_\lambda(x)) \in  \bbL_{\scal{\cdot, \cdot}}(G)$ is a smooth diffeomorphism, it sends null sets of lines to null sets of lines.
\end{remark}

For any given choice of scalar product $\scal{\cdot, \cdot}$ on $\g_1$, it will be convenient for our purposes to work with particular choices of the notions defined above. Namely, let $\scal{\cdot, \cdot}$ be scalar product on $\g_1$. First, given $X\in \bbS(\g_1)$, we equip $N_{X}$ with a Haar measure $\mu_{X}$ normalized in such a way that the restriction of $\Phi_X$ to $N_{X} \times \R$ is a measure preserving smooth diffeomorphism from $(N_{X} \times \R, \mu_{X} \otimes dt)$ to $(G,\mu)$. This is indeed always possible as we briefly explain now. Complete $X$ into a basis $(X,X_2,\dots,X_n)$ of $\g$ adapted to the stratification and so that $(X,X_2,\dots,X_r)$ is an orthonormal basis of $(\g_1,\scal{\cdot,\cdot})$. Here $n=\dim \g$ and $r=\dim\g_1$. Using exponential coordinates of the first kind associated to such a basis, one can identify $G$ with $\R^n$. As a well known fact, it follows from the Baker-Campbell-Hausdorff (BCH) formula that the $n$-dimensional Lebesgue measure is a Haar measure on $G\simeq \R^n$. Similarly, we identify $N_{X}$ with $\R^{n-1}$ using exponential coordinates of the first kind associated to the basis $(X_2,\dots,X_n)$ of $N_X$ and it also follows from the BCH formula that the $(n-1)$-dimensional Lebesgue measure is a Haar measure on $N_{X} \simeq \R^{n-1}$. Using once again the BCH formula, it is a routine computation to verify that the smooth map 
\begin{equation*}
\begin{split}
\Phi_{X}|_{N_{X} \times \R}: N_{X} \times \R \simeq \R^{n-1}\times \R  &\rightarrow  G \simeq \R^n\\
(\exp\sum_{j=2}^n x_j X_j,t) &\mapsto (\exp\sum_{j=2}^n x_j X_j)\cdot \exp(tX)
\end{split}
\end{equation*}
has Jacobian determinant identically equal to 1. 

Next, we denote by $\sigma_{\scal{\cdot,\cdot}}$ the standard measure on $\bbS(\g_1)$ and define the measure $\calN_{\scal{\cdot,\cdot}}$ on $\bbL_{\scal{\cdot, \cdot}}(G)$ setting
\begin{equation*}
\calN_{\scal{\cdot,\cdot}}(A) = \int_{\bbS(\g_1)} \mu_{X}(\{x\in N_{X}:\, (X,x) \in A\}) \, d\sigma_{\scal{\cdot,\cdot}}(X)
\end{equation*}
for any Borel subset $A$ of $\bbL_{\scal{\cdot, \cdot}}(G)$. It can easily be seen from the previous construction that $\calN_{\scal{\cdot,\cdot}}$ is a smooth measure with positive density on $\bbL_{\scal{\cdot, \cdot}}(G)$. 

Finally, given $L=(X,x) \in \bbL_{\scal{\cdot, \cdot}}(G)$, the map $\Phi_X(x,\cdot)$ is a homeomorphism from $\R$ to $L$, and we set $\calL^1_{L,\scal{\cdot, \cdot}} = \Phi_X(x,\cdot)_*\calL^1$. A non empty connected strict subset of $L$ that has connected complement in $L$ is called a half-line in $L$. Equivalently, a half-line in $L$ is a set of the form $\Phi_X(x,I)$ where $I$ is a half-line in $\R$, i.e., $I$ is an unbounded strict subinterval of $\R$.

\begin{definition} \label{def:monotone-direction} Given a scalar product $\scal{\cdot, \cdot}$ on $\g_1$, we say that $X\in \bbS(\g_1)$ is a monotone direction for a subset $E$ of $G$ if for $\mu_X$-a.e. $x\in N_X$ the line $L=(X,x)\in \bbL_{\scal{\cdot, \cdot}}(G)$ is monotone for $E$, where this latter property can be rephrased saying that the following equivalent conditions hold:
\begin{itemize}

\item[(i)] $\1_E(\Phi_X(y,t)) = \1_{I} (t)$ for $\calL^1$-a.e.~$t\in\R$ where either $I \in \{\emptyset,\R\}$ or $I$ is a half-line in $\R$,

\item[(ii)] either $\calL^1_{L,\scal{\cdot, \cdot}}(E \cap L) = 0$, or $\calL^1_{L,\scal{\cdot, \cdot}}(E^c \cap L) = 0$, or there is a half-line $I$ in $L$ such that $\calL^1_{L,\scal{\cdot, \cdot}}(( E\triangle I) \cap L) = 0$.
\end{itemize}
\end{definition}

As an immediate consequence of the previous constructions, we get the following characterization of monotone sets.

\begin{lemma} A $\mu$-measurable subset $E$ of $G$ is monotone if and only if $\sigma_{\scal{\cdot, \cdot}}$-a.e. $X\in \bbS(\g_1)$ is a monotone direction for $E$ for some, equivalently any, scalar product $\scal{\cdot, \cdot}$ on $\g_1$.
\end{lemma}

\begin{remark} Constant-normal sets are monotone. Constant-normal sets play an important role in the study of sets with locally finite perimeter, see for instance \cite{MR1871966,MR1984849,MR3353698}. They are fully understood in step-2 and type-$\star$ Carnot groups \cite{MR1984849,MR3330910} and partially in a few other specific examples of Carnot groups, see for instance \cite{MR3078944}. They have been studied in the general setting of an arbitrary Carnot group in \cite{MR4228617} and one should in particular compare \cite[Section~3]{MR4228617} with Theorem~\ref{thm:minimizer-for-perimeter}. We would like to stress here that the study of monotone sets goes beyond that of constant-normal ones and is in many respects more delicate. Recall that $E$ has constant-normal if  $E$ is a $\mu$-measurable subset of $G$ and there is a closed half-space $W$ of $\g_1$ such that  for any $X\in W$ we have $X\1_E \geq 0$ in the sense of distribution. Equivalently, a $\mu$-measurable subset $E$ of $G$ has constant-normal if and only if for some, equivalently any, scalar product $\scal{\cdot, \cdot}$ on $\g_1$, there is $X_1 \in \bbS(\g_1)$ such that the following holds for any $X\in \bbS(\g_1)$ such that $\scal{X_1,X} \geq 0$: for $\mu_X$-a.e.~$x\in N_X$ the function $t\in \R \mapsto \1_E(\Phi_X(x,t))$ agrees $\calL^1$-a.e. with a nondecreasing function,  equivalently, there is $t_x\in [-\infty,+\infty]$ such that $E\cap \Phi_X(x,\R)$ is equivalent to $\Phi_X(x,(t_x,+\infty))$, see \cite[Proposition~2.8]{MR4228617}.\footnote{We warm the reader not to be confused by the terminology "monotone direction" used in several papers where constant-normal sets are considered and where it does not have the same meaning that the one introduced by Cheeger and Kleiner \cite{MR2729270} when dealing with the larger class of monotone sets, which is the one we follow in the present paper. More explicitly, recall from Definitions~\ref{def:monotone-line} and~\ref{def:monotone-direction}, see also \cite[Definition~5.3]{MR2729270}, that in the present paper a direction $X\in \bbS(\g_1)$ is said to be monotone for a $\mu$-measurable set $E$ if for $\mu_X$-a.e.~$x\in N_X$ the function $t\in \R \mapsto \1_E(\Phi_X(x,t))$ agrees $\calL^1$-a.e. with a \textit{monotone} function, while in e.g. \cite{MR3353698,MR4228617} a direction $X\in \bbS(\g_1)$ is said to be monotone for a $\mu$-measurable set $E$ if $X\1_E\geq 0$, meaning that for $\mu_X$-a.e.~$x\in N_X$ the function $t\in \R \mapsto \1_E(\Phi_X(x,t))$ agrees $\calL^1$-a.e. with a \textit{nondecreasing} function.} If $G$ is abelian, i.e., $s=1$ and the Carnot group structure of $G$ is therefore nothing but that of a finite dimensional real vector space, monotone sets are equivalent to half-spaces \cite[Lemma~64]{MR3815462} and have constant-normal. On the contrary, when $G$ is non abelian, i.e., when $s\geq 2$, having constant-normal is a much more restrictive condition that monotonicity. As an easy example, it is not hard to verify that in an arbitrary Carnot group of step $s\geq 2$, for any $Z\in \g_2$, $Z\not=0$, the half-space $E_Z = \{ \exp Y \in G:\, Y\in \g,\, \scal{Y,Z} > 0\}$ is monotone but does not have constant-normal. In this example $\scal{\cdot, \cdot}$ denotes a scalar product defined on the whole $\g$ and  adapted to the stratification. Moreover, if $G$ is of M\'etivier's type, i.e., if the map $Y\in\g_1 \mapsto [X,Y] \in [\g,\g]$ is surjective for any $X\in\g_1^*$, then there is no $X\in \g_1^*$ such that $X\1_{E_Z}$ has constant sign and there is no open cone $C$ with respect to which $E_Z$ satisfies the $C$-cone condition considered in \cite{MR4228617}.  We also recall that in an arbitrary Carnot group sublevel sets of horizontally affine functions are precisely monotone sets that may not be equivalent to half-spaces, and this does already happen  in some step-2 cases~\cite{https://doi.org/10.48550/arxiv.2004.08129,https://doi.org/10.48550/arxiv.2106.13490}. On the contrary, constant-normal subsets of step-2 Carnot groups are known \cite{MR1984849} to be equivalent to vertical half-spaces, i.e., sets of the form $ \{ \exp Y \in G:\, Y\in \g,\, \scal{Y,X} > 0\}$ for some $X\in \g_1^*$.
\end{remark}


\subsection{Left-invariant homogeneous distances} \label{subsect:distances}

A distance $d$ on $G$ is said to be homogeneous if $d(\delta_\lambda(x), \delta_\lambda(y)) = \lambda d(x,y)$ for any $x,y\in\g$, $\lambda>0$. Here $\delta_\lambda:\g \rightarrow \g$ denotes the dilation with factor $\lambda$ (see Remark~\ref{rmk:monotonicity-dilations}). It is well known that the topology induced by any left-invariant homogeneous distance coincide with the topology induced by the Lie group structure of $G$, see for instance~\cite[Proposition~2.26]{MR3943324}. We recall below some well known facts that will be relevant for our purposes.

\smallskip

For any left-invariant homogeneous distance $d$ on $G$, Haar measures are doubling with respect to $d$. They are more precisely $Q$-uniform, i.e., 
\begin{equation} \label{e:uniform-Haar-meas}
\mu(B_d(x,r)) = c_d r^Q 
\end{equation}
for any $x\in G$, $r>0$, where $c_d  = \mu(B_d(0,1))>0$ and $Q=\sum_{i=1}^s i\dim \g_i$ denotes the homogeneous dimension of $G$.

\smallskip

Given a left-invariant homogeneous distance $d$ on $G$, a subset $E$ of $G$, and $t>0$, we set 
\begin{equation*}
E_{d,t} = \left\{x\in G:\, \lim_{r\downarrow 0} \frac{\mu(E\cap B_d(x,r))}{\mu(B_d(x,r))} =t \right\}.
\end{equation*}
We recall that if $E$ is $\mu$-measurable then any two left-invariant homogeneous distances $d$ and $\rho$ yield the same sets $E_{d,0} = E_{\rho,0}$ and $E_{d,1} = (E^c)_{d,0} = (E^c)_{\rho,0}=E_{\rho,1}$ that we denote by $E_0$ and $E_1$ respectively. This indeed follows from the doubling property of Haar measures with respect to left-invariant homogeneous distances together with the fact that any two such distances are biLipschitz equivalent. We also denote by $\partial^* E = G \setminus (E_0 \cup E_1)$ the essential boundary of a set $E$.

\smallskip

As another consequence of the doubling property of Haar measures with respect to left-invariant homogeneous distances, we recall that if $E$ is $\mu$-measurable then both $E_1$ and $G\setminus E_0$ are measure-theoretic representatives of $E$, i.e.,
\begin{equation} \label{e:lebesgue-representative}
\mu(E\triangle E_1) = 0 \quad \text{and} \quad \mu(E^c \triangle E_0) = 0.
\end{equation}

\smallskip

\noindent\textit{Example.} Given a scalar product $\scal{\cdot,\cdot}$ on $\g_1$, we denote by $\bfX$ the collection of all orthonormal bases of $(\g_1,\scal{\cdot,\cdot})$ and by $d_\bfX$ the sub-Riemannian distance induced by some, equivalently any, element in $\bfX$. It is well known that $d_\bfX$ is left-invariant and homogeneous, see for instance~\cite{MR3587666}.

\begin{remark} Given any scalar product $\scal{\cdot,\cdot}$ on $\g_1$, a subset $A$ of a line $L=(X,x)\in \bbL_{\scal{\cdot,\cdot}}(G)$ is a null subset of $L$ if and only if for some, equivalently any, left-invariant homogeneous distance $d$ on $G$ we have $\calH^1_d(A)= 0$ where $\calH^1_d$ denotes the 1-dimensional Hausdorff measure induced by $d$. Indeed, let $d$ be a left-invariant homogeneous distance on $G$. Then $d(\Phi_X(x,t), \Phi_X(x,s)) = c_{d,X} |t-s|$ for all $s,t \in \R$ where $c_{d,X} = d(0,\exp X)$ and therefore $\calL^1_{L,\scal{\cdot, \cdot}} = c_{d,X} \calH^1_d\lfloor_L$.
\end{remark}

\section{Monotone sets and local minimizers for the perimeter}  \label{sect:monotone-local-min-perimeter}

Given $X\in \g$ and a $\mu$-measurable subset $E$ of  $G$, we denote by $X\1_E$ the distribution given by 
\begin{equation*}
\scal{X\1_E,\varphi} = - \int_E X\varphi \, d\mu, \quad \varphi \in C^\infty_c(G).
\end{equation*}

We recall that a subset $E$ of $G$ is said to have locally finite perimeter if $E$ is a $\mu$-measurable subset of $G$ and  $X\1_E$ is a Radon measure for any $X\in \g_1$.

\smallskip

We first prove Theorem~\ref{thm:monotone-implies-locally-finite-perimeter} which asserts that monotone sets have locally finite perimeter. In the rest of this paper, given $E, F \subset G$, we write $E\subset \subset F$ to mean that $E$ is a relatively compact subset of $F$.

\begin{proof}[Proof of Theorem~\ref{thm:monotone-implies-locally-finite-perimeter}]
Let $E$ be monotone. Fix a scalar product $\scal{\cdot,\cdot}$ on $\g_1$. Let $X \in \bbS(\g_1)$ be a monotone direction for $E$. For $\mu_{X}$-a.e.~$x\in N_{X}$ the line $(X,x)\in \bbL_{\scal{\cdot,\cdot}}(G)$ is monotone for $E$, i.e., $\1_E(\Phi_X(x,t)) = \1_{I_x}(t)$ for a.e.~$t\in \R$ where either $I_x\in \{\emptyset,\R\}$ or $I_x$ is a half-line in $\R$. Let $\Omega\subset \subset G$ be open, $K \subset N_{X}$ be compact, and $a,b \in \R$, $a<b$, be such that $\Phi_X^{-1}(\Omega) \cap (N_{X} \times \R) \subset K \times [a,b]$. Let $\varphi \in C^1_c(\Omega)$. Since $\Phi_X$ is a measure preserving diffeomorphism from $(N_{X} \times \R, \mu_{X} \otimes dt)$ to $(G,\mu)$, we have 
\begin{equation*}
\begin{split}
\int_E   X\varphi \, d\mu &= \int_K \int_a^b \1_E(\Phi_X(x,t)) X\varphi(\Phi_X(x,t)) \, dt \, d\mu_{X}(x) \\
&= \int_K \int_{ (a,b) \cap I_x} X\varphi(\Phi_X(x,t)) \, dt \, d\mu_X(x).
\end{split}
\end{equation*}
Since either $I_x\in \{\emptyset,\R\}$ or $I_x$ is a half-line, $(a,b) \cap I_x$ is either empty or a subinterval of $(a,b)$ such that $\inf (a,b) \cap I_x = a$ or $\sup (a,b) \cap I_x = b$. Since 
\begin{equation*}
\int_c^d X\varphi(\Phi_X(x,t)) \, dt = \varphi(\Phi_X(x,d)) - \varphi(\Phi_X(x,c))
\end{equation*}
for any $c,d \in \R$, we get that 
\begin{equation*}
| \int_{(a,b) \cap I_x} X\varphi(\Phi_X(x,t)) \, dt | \leq \sup|\varphi|.
\end{equation*}
Therefore
\begin{equation*}
|\int_E  X\varphi \, d\mu | \leq \mu_{X} (K)  \sup|\varphi|,
\end{equation*}
i.e., $X\1_E$ is a Radon measure. Since $\sigma_{\scal{\cdot,\cdot}}$-a.e.~direction in $\bbS(\g_1)$ is monotone for $E$ and $\{X\in \g_1: \, X\1_E \text{ is a Radon measure}\}$ is a linear subspace of $\g_1$, we get that $X\1_E$ is a Radon measure for any $X\in \g_1$, i.e., $E$ has locally finite perimeter.
\end{proof}

Given a scalar product $\scal{\cdot,\cdot}$ on $\g_1$ with associated collection $\bfX$ of orthonormal bases of $(\g_1,\scal{\cdot,\cdot})$, the $\bfX$-perimeter of a $\mu$-measurable set $E$ in an open set $\Omega$ is defined as
\begin{equation*} 
\Per_{\bfX}(E,\Omega) = \sup \left\{ \int_E \sum_{i=1}^r X_i\varphi \, d\mu: \varphi \in C^1_c(\Omega,\R^r), \norm{\varphi}_{\infty} \leq 1 \right\}
\end{equation*}
for some, equivalently any, $(X_1,\dots,X_r) \in \bfX$. 

\smallskip

The value of the $\bfX$-perimeter depends on the choice of the scalar product $\scal{\cdot,\cdot}$ on $\g_1$. However, for any two choices of scalar product on $\g_1$ with associated collection $\bfX$ and $\mathbf{X'}$ of orthonormal bases, there are $c,C>0$ such that $c \Per_{\bfX}(E,\Omega) \leq \Per_{\mathbf{X'}}(E,\Omega) \leq C \Per_{\bfX}(E,\Omega)$ for any $\mu$-measurable set $E$ and any open set $\Omega$, where these inequalities hold in $[0,+\infty]$.

\smallskip

It is well known that a $\mu$-measurable set $E$ has locally finite perimeter if and only if for any $\Omega \subset \subset G$ open we have $\Per_{\bfX}(E,\Omega)<+\infty$ for some, equivalently any, scalar product $\scal{\cdot,\cdot}$ on $\g_1$. Furthermore, when $E$ has locally finite perimeter, the set function $\Omega \mapsto \Per_{\bfX}(E,\Omega)$ is the restriction to the relatively compact open subsets of $G$ of a Radon measure that we denote by $\Per_{\bfX}(E,\cdot)$. We refer to \cite{MR3587666} and the references therein for more information about the $\bfX$-perimeter.

\smallskip

\begin{definition} \label{def:local-minimizer-perimeter} Given a scalar product  $\scal{\cdot,\cdot}$ on  $\g_1$ with associated collection $\bfX$ of orthonormal bases of $(\g_1,\scal{\cdot,\cdot})$, we say that a subset $E$ of $G$ is a local minimizer for the $\bfX$-perimeter if $E$ has locally finite perimeter and for any open set $\Omega \subset \subset G$ we have $\Per_{\bfX}(E,\Omega) \leq \Per_{\bfX}(F,\Omega)$ for any $\mu$-measurable set $F$ such that $E\triangle F \subset\subset \Omega$.
\end{definition}

\begin{remark} A local minimizer for the $\bfX$-perimeter for some given choice of scalar product on $\g_1$ with associated collection $\bfX$ of orthonormal bases may not be a local minimizer for the $\mathbf{X'}$-perimeter for any other different choice of scalar product on $\g_1$ with associated collection $\mathbf{X'}$ of orthonormal bases. Note that  Theorem~\ref{thm:monotone-implies-locally-finite-perimeter} shows that monotone sets are local minimizers for the perimeter in the stronger sense that they are local minimizer of the $\bfX$-perimeter for any choice of scalar product on $\g_1$. This is in turn coherent with the fact that monotonicity is independent of the choice of a scalar product on $\g_1$.
\end{remark}

From now on in this section, we let $\scal{\cdot,\cdot}$ be a given scalar product on $\g_1$ with associated collection $\bfX$ of orthonormal bases of $(\g_1,\scal{\cdot,\cdot})$. The proof of Theorem~\ref{thm:monotone-minimizer-for-perimeter} saying that monotone sets are local minimizers for the $\bfX$-perimeter will be obtained as a consequence of the kinematic formula \cite[Proposition~3.13]{MR2165404} that relates the $\bfX$-perimeter to 1-dimensional perimeter on lines. 

Given a line $L=(X,x)\in \bbL_{\scal{\cdot,\cdot}}(G)$, a $\calL^1_{L,\scal{\cdot, \cdot}}$-measurable subset $E$ of $L$, and an open subset $U$ of $L$, set
\begin{equation*}
\Per_L(E,U) = \Per_\R(\Phi_X(x,\cdot)^{-1}(E),\Phi_X(x,\cdot)^{-1}(U))
\end{equation*}
where $\Per_\R$ denotes the usual perimeter for $\calL^1$-measurable subsets of $\R$.

Note that it follows from well known properties of finite perimeter subsets of $\R$ (see \cite{MR1857292}) that a $\mu$-measurable set $E$ is monotone if and only if for $\calN_{\scal{\cdot,\cdot}}$-a.e.~$L\in\bbL_{\scal{\cdot,\cdot}(G)}$ the 1-dimensional perimeter $\Per_L(E\cap L,L)$ is either 0 or 1. Note incidentally that if $E$ is $\mu$-measurable then $E\cap L$ is $\calL^1_{L,\scal{\cdot, \cdot}}$-measurable for $\calN_{\scal{\cdot,\cdot}}$-a.e.~$L\in\bbL_{\scal{\cdot,\cdot}}(G)$.

It is proved in \cite{MR2165404} that there is $c_\bfX>0$ such that for any open set $\Omega \subset G$ and  any $\mu$-measurable set $E$ with $\Per_\bfX(E,\Omega) <+\infty$ the map $L \mapsto \Per_L(E\cap L, \Omega \cap L)$ lies in $L^1(\bbL_{\scal{\cdot,\cdot}(G)},\calN_{\scal{\cdot,\cdot}})$ and one has 
\begin{equation} \label{e:kinematic-formula}
\Per_\bfX(E,\Omega) =  c_\bfX \int_{\bbL_{\scal{\cdot,\cdot}}(G)} \Per_L(E\cap L, \Omega \cap L) \, d\calN_{\scal{\cdot,\cdot}}(L).
\end{equation}

\smallskip

\begin{proof}[Proof of Theorem~\ref{thm:monotone-minimizer-for-perimeter}]
Let $E$ be monotone. We know from Theorem~\ref{thm:monotone-implies-locally-finite-perimeter} that $E$ has locally finite perimeter and thus it remains to verify that $E$ satisfies the minimality condition given in Definition~\ref{def:local-minimizer-perimeter}.  Let $\Omega \subset \subset G$ be open and $F$ be $\mu$-measurable such that $E\triangle F \subset\subset \Omega$. Assume that $\Per_\bfX(F,\Omega)<+\infty$ since otherwise we obviously have $\Per_\bfX(E,\Omega)\leq \Per_\bfX(F,\Omega)$. By the kinematic formula~\eqref{e:kinematic-formula} it is sufficient to prove that
\begin{equation} \label{e:minimizer}
\Per_L(E\cap L, \Omega\cap L) \leq \Per_L(F\cap L, \Omega \cap L)
\end{equation}
for $\calN_{\scal{\cdot,\cdot}}$-a.e. $L\in\bbL_{\scal{\cdot,\cdot}}(G)$.

Let $L= (X,x)\in \bbL_{\scal{\cdot,\cdot}}(G)$ be such that $F\cap L$ is $\calL^1_{L,\scal{\cdot,\cdot}}$-measurable and $L$ is monotone for $E$. Since both conditions hold for $\calN_{\scal{\cdot,\cdot}}$-a.e.~line, we only need to verify \eqref{e:minimizer} for such a line. If $\Per_L(E\cap L, \Omega\cap L) =0$ then obviously $\Per_L(E\cap L, \Omega\cap L) \leq \Per_L(F\cap L, \Omega \cap L)$. If $\Per_L(E\cap L, \Omega\cap L)=1$ then $E\cap L$ is equivalent to a half-line $I$ in $L$ whose endpoint $x_I$ is contained in $\Omega$. Set $\Phi(\cdot) = \Phi_X(x,\cdot)$ and recall that $\Phi$ is a homeomorphism from $\R$ to $L$. Let $J_L$ denote the connected component of $\Phi^{-1}(\Omega)$ containing $t_I = \Phi^{-1}(x_I)$. Set $E_L = \Phi^{-1}(E)$ and $F_L = \Phi^{-1}(F)$. Since $(E\triangle F) \subset \subset \Omega$,  we have $(E_L \triangle F_L) \cap J_L \subset \subset J_L$. Let $U_L \subset \subset J_L$ be open and such that $(E_L \triangle F_L) \cup \{t_I\} \subset U_L$. Since $J_L$ is a non empty open bounded interval of $\R$, both $(-\infty, t_L) \cap J_L \setminus \overline{U_L}$ and $(t_L,+\infty) \cap J_L \setminus \overline{U_L}$ are open and non empty. It follows that 
\begin{equation*}
\begin{split}
\min &\left\{ \calL^1 (F_L \cap J_L), \calL^1 (F_L^c \cap J_L ) \right\} \\
 &\geq \min \left\{ \calL^1 (F_L \cap J_L \setminus \overline{U_L}), \calL^1 (F_L^c \cap J_L \setminus \overline{U_L}) \right\} \\
&= \min \left\{ \calL^1 ((t_L,+\infty) \cap J_L \setminus \overline{U_L}), \calL^1 ((-\infty, t_L) \cap J_L \setminus \overline{U_L}) \right\} >0
\end{split}
\end{equation*}
which implies that $\Per_\R(F_L,J_L)\geq 1$. Since $\Per_L(F\cap L, \Omega \cap L) \geq \Per_\R(F_L,J_L)$, this concludes the proof.
\end{proof}

We now turn to the study of local miminizers for the $\bfX$-perimeter. We recall that given a Borel measure $\nu$ on $G$, we denote by $\spt \nu$ its support. Given a $\mu$-measurable set $E$, we denote by $\spt_\mu(E)$ the support of the restriction of $\mu$ to $E$,
\begin{equation*}
\begin{split}
\spt_\mu (E) &= G \setminus \bigcup \{U\subset G:\, U \text{ open},\, \mu(E \cap U)=0\}\\
&= \{x\in G:\, \mu(E\cap B_d(x,r)) >0 \,\, \forall r>0\},
\end{split}
\end{equation*} 
we set 
\begin{equation*}
\begin{split}
\Int_\mu(E) &= G \setminus \spt_\mu (E^c)\\
&=\{x\in G:\exists r>0 \text{ such that } \mu(E^c\cap B_d(x,r)) =0\}
\end{split}
\end{equation*} 
and 
\begin{equation*} 
\begin{split}
\partial_\mu E &= \spt_\mu (E) \cap \spt_\mu (E^c) \\
&=\{x\in G: \min\{\mu(E\cap B_d(x,r)),\mu(E^c\cap B_d(x,r))\} >0 \,\, \forall r>0\}.
\end{split}
\end{equation*} 
Here $d$ denotes any left-invariant homogeneous distance, or more generally, any distance $d$ whose topology coincides with the topology induced by the Lie group structure of $G$.

\smallskip

To put Theorem~\ref{thm:minimizer-for-perimeter} into perspective, note first that for any $\mu$-measurable set $E$ and any left-invariant homogeneous distance $d$ we have
\begin{equation*}
\partial^* E = \{x\in G:\, \limsup_{r\downarrow 0} h_d(x,r) >0\} \subset \partial_\mu E = \{x\in G:\, \ h_d(x,r) >0 \,\, \forall r>0\}
\end{equation*}
where 
\begin{equation} \label{e:def-h}
h_d(x,r) = \min\left\{\frac{\mu(E\cap B_d(x,r))}{\mu(B_d(x,r))}, \frac{\mu(E^c \cap B_d(x,r))}{\mu(B_d(x,r))}\right\}.
\end{equation} 
This inclusion is strict in general and it may even happen that $\mu(\partial_\mu E \setminus \partial^* E) = +\infty$ already for sets with locally finite perimeter. Recall indeed that if $E$ has locally finite perimeter then the measure $\Per_\bfX(E,\cdot)$ is concentrated on $\partial^* E$ and  the restriction of $\calH^{Q-1}_d$ to $\partial^* E$ is a locally finite measure for any left-invariant homogeneous distance $d$ on $G$. We refer to \cite[Theorems~4.2 and 4.3]{MR1823840} for more precise statements that should be compared with \eqref{e:meas-theoretic-boundary-uniform-lower-densities} and \eqref{e:boundaries}. Recall also that if $E$ has locally finite perimeter then $\spt \Per_\bfX(E,\cdot) = \partial_\mu E$. Indeed, on the one hand it follows from the relative isoperimetric inequality \cite[Theorem~1.18]{MR1404326} that $\partial_\mu E \subset \spt \Per_\bfX(E,\cdot)$. On the other hand, one knows that $\Per_\bfX(E,\cdot)$ is concentrated on $\partial^* E$ and hence on $\partial_\mu E$. Since $\partial_\mu E$ is closed and $\spt \Per_\bfX(E,\cdot)$ is the smallest closed set on which $\Per_\bfX(E,\cdot)$ is concentrated, we get that $\spt \Per_\bfX(E,\cdot) = \partial_\mu E$. We stress that there are examples of sets $E$ with locally finite perimeter for which $\mu(\spt \Per_\bfX(E,\cdot)) =+\infty$, i.e., $\mu(\partial_\mu E \setminus \partial^* E) = +\infty$. Note also that any of the following topological boundaries, $\partial \Int_\mu(E)$, $\partial \spt_\mu (E)$, $\partial E_1$, and $\partial E_0$ is a subset of $\partial_\mu E$, which is in turn contained in the topological boundary of $E$. Once again, any of these inclusions may be strict. These observations should be compared with \eqref{e:meas-theoretic-boundary-uniform-lower-densities} and \eqref{e:boundaries}.
 
\smallskip

Next, note that for any $\mu$-measurable set $E$ we have 
\begin{equation*}
\Int(E) \subset \Int_\mu(E) \subset E_1 \subset G\setminus E_0 \subset \spt_\mu (E) \subset \overline{E}
\end{equation*}
and any of these inclusions may be strict. It may indeed happen that $\mu(E_1 \setminus \Int_\mu(E))=+\infty$, or $\mu(\spt_\mu(E) \cap E_0) =+\infty$, already for sets $E$ with locally finite perimeter. This should be compared with \eqref{e:interior-and-support}.

\smallskip

Concerning topological properties, neither $E_1$, nor $E_0$, nor $\partial^* E$, are open nor closed in general, while $\Int_\mu(E)$ is open and $\spt_\mu(E)$ and $\partial_\mu E$ are closed. Therefore \eqref{e:interior-and-support} and \eqref{e:boundaries} contain non trivial topological informations as well. Note also that \eqref{e:regular-open-closed} implies that $\Int_\mu (E)$ coincides with the interior of its closure and $\spt_\mu (E)$ with the closure of its interior.

\smallskip

The main step for the proof of Theorem~\ref{thm:minimizer-for-perimeter} is given in the following lemma. Recall that $d_\bfX$ denotes the sub-Riemannian distance induced by any element in $\bfX$.

\begin{lemma} \label{lem:small-density} There is $\epsilon_\bfX>0$ such that for any local minimizer $E$ for the $\bfX$-perimeter, $x\in G$, $r>0$, one has
\begin{align}
\mu(E\cap B_{d_\bfX}(x,r)) \leq \epsilon_\bfX \mu(B_{d_\bfX}(x,r)) &\Rightarrow \mu(E\cap B_{d_\bfX}(x,r/2)) = 0, \label{e:1}\\
\mu(E^c\cap B_{d_\bfX}(x,r)) \leq \epsilon_\bfX \mu(B_{d_\bfX}(x,r)) &\Rightarrow \mu(E^c\cap B_{d_\bfX}(x,r/2)) = 0. \label{e:2}
\end{align}
\end{lemma}

\begin{proof}
Let $\epsilon>0$ to be fixed small later. Let a local minimizer $E$ for the $\bfX$-perimeter, $x\in G$, $r>0$, be given, and assume that $\mu(E\cap B_{d_\bfX}(x,r)) \leq \epsilon \mu(B_{d_\bfX}(x,r))$. Since $E$ is a local minimizer for the $\bfX$-perimeter, we have for $t\in (0,r)$
\begin{equation*}
\begin{split}
&\Per_{\bfX}( E,B_{d_\bfX}(x,r)) \leq \Per_\bfX (E\setminus B_{d_\bfX}(x,t),B_{d_\bfX}(x,r))\\
&= \Per_\bfX(E,B_{d_\bfX}(x,r)\setminus \overline{B_{d_\bfX}(x,t)}) + \Per_\bfX(E\setminus B_{d_\bfX}(x,t),\partial B_{d_\bfX}(x,t))\\
&\leq \Per_\bfX(E,B_{d_\bfX}(x,r)) - \Per_\bfX(E,B_{d_\bfX}(x,t)) + \Per_\bfX(E\setminus B_{d_\bfX}(x,t),\partial B_{d_\bfX}(x,t)),
\end{split}
\end{equation*}
i.e., 
\begin{equation*}
\Per_\bfX(E,B_{d_\bfX}(x,t)) \leq \Per_\bfX (E\setminus B_{d_\bfX}(x,t),\partial B_{d_\bfX}(x,t)).
\end{equation*}
For $t>0$ set $m(t) = \mu(E\cap B_{d_\bfX}(x,t))$. It follows from the previous inequality, the relative isoperimetric inequality \cite[Theorem~1.18]{MR1404326}, and \cite[Lemma~3.5]{MR1823840}, that 
\begin{equation} \label{e:I1}
\min\{\mu(E\cap B_{d_\bfX}(x,t)), \mu(E^c\cap B_{d_\bfX}(x,t))\}^{(Q-1)/Q} \lesssim m'(t)
\end{equation}
for a.e.~$t\in (0,r)$. Here $\lesssim$ means that the inequality holds up to a positive multiplicative constant that depends only on $\bfX$. Since $\mu$ is a doubling measure with respect to $d_\bfX$, we have for $t\in (r/2,r)$
\begin{equation*} 
\mu(E\cap B_{d_\bfX}(x,t)) \leq \mu(E\cap B_{d_\bfX}(x,r)) \leq \epsilon \mu(B_{d_\bfX}(x,r)) \lesssim \epsilon \mu(B_{d_\bfX}(x,t)).
\end{equation*}
Therefore, if $\epsilon>0$ is choosen small enough, we have $\min\{\mu(E\cap B_{d_\bfX}(x,t)), \mu(E^c\cap B_{d_\bfX}(x,t))\} = \mu(E\cap B_{d_\bfX}(x,t))=m(t)$ for all $t\in (r/2,r)$ and we get from \eqref{e:I1} 
\begin{equation*}
m(t)^{(Q-1)/Q} \lesssim m'(t) \quad \text{for a.e.}~t\in (r/2,r).
\end{equation*}
Assume that $m(r/2)>0$ otherwise we are done. Then $1\lesssim m(t)^{(1-Q)/Q} m'(t)$ for a.e.~$t\in (r/2,r)$. Integrating over $(r/2,r)$ and using \eqref{e:uniform-Haar-meas} we get
\begin{equation*}
r \lesssim m(r)^{1/Q} - m(r/2)^{1/Q} \lesssim \epsilon^{1/Q} \mu(B_{d_\bfX}(x,r))^{1/Q} \lesssim \epsilon^{1/Q} r.
\end{equation*}
This gives a contradiction if $\epsilon>0$ is choosen small enough and concludes the proof of \eqref{e:1}. Since a set $E$ is a local minimizer for the $\bfX$-perimeter if and only if $E^c$ is a local minimizer for the $\bfX$-perimeter, \eqref{e:2} follows as well.
\end{proof}

\begin{corollary} \label{cor:main} For any left-invariant homogeneous distance $d$ on $G$ there is $\epsilon >0$ such that for any local minimizer $E$ for the $\bfX$-perimeter one has
\begin{equation} \label{e:description-meas-theoretic-boundary}
\partial_\mu E = \left\{x\in G:\, \min\left\{\frac{\mu(E\cap B_d(x,r))}{\mu(B_d(x,r))}, \frac{\mu(E^c \cap B_d(x,r))}{\mu(B_d(x,r))}\right\} >\epsilon \, \,\forall r>0\right\}.
\end{equation}
\end{corollary}

\begin{proof}
Let $E$ be a local minimizer for the $\bfX$-perimeter. Recall that 
\begin{equation*}
\partial_\mu E = \{x\in G:\, h_d(x,r) >0 \, \,\forall r>0\}
\end{equation*}
for any left-invariant homogeneous distance $d$ on $G$, where $h_d(x,r)$ is given by \eqref{e:def-h}. Then, as an immediate consequence of Lemma~\ref{lem:small-density}, one gets the existence of $\epsilon_\bfX>0$ such that 
\begin{equation*}
\partial_\mu E = \{x\in G:\, h_{d_\bfX}(x,r) >\epsilon_\bfX \, \,\forall r>0\}.
\end{equation*} 
Any two left-invariant homogeneous distances on $G$ are biLipschitz equivalent and $\mu$ is doubling with respect to $d_\bfX$. Therefore, for any left-invariant homogeneous distance $d$ on $G$, there is $\epsilon >0$ such that
\begin{equation*}
\{x\in G:\, h_{d_\bfX}(x,r) >\epsilon_\bfX \, \,\forall r>0\} \subset \{x\in G:\, h_d(x,r) > \epsilon \, \,\forall r>0\},
\end{equation*} 
which in turn implies \eqref{e:description-meas-theoretic-boundary}.
\end{proof}

\begin{lemma} \label{lem:consequence-of-corollary}
Let $E$ be $\mu$-measurable and assume that \eqref{e:description-meas-theoretic-boundary} holds for some left-invariant homogeneous distance $d$ on $G$ and some $\epsilon>0$. Then
\begin{align}
& \Int_\mu(E) = E_1 \text{ and }  \spt_\mu(E) = G \setminus E_0, \label{e:lem-interior-and-support} \\
& \partial_\mu E = \partial^* E =\partial \Int_\mu(E) = \partial \spt_\mu(E), \label{e:lem-boundaries}\\
& \overline{\Int_\mu(E)} = \spt_\mu(E) \text{ and } \Int(\spt_\mu(E)) = \Int_\mu(E). \label{e:lem-regular-open-closed}
\end{align}
\end{lemma}

\begin{proof}
To prove \eqref{e:lem-interior-and-support}, recall first that $\Int_\mu(E) \subset E_1$. Conversely, note that 
\begin{equation*}
E_1 \subset \{x\in G:\, \lim_{r\downarrow 0} h_d(x,r)=0 \}
\end{equation*}
where $h_d(x,r)$ is given by \eqref{e:def-h}. Then we get from \eqref{e:description-meas-theoretic-boundary} that $E_1 \cap \partial_\mu E = \emptyset$. Since $E_1 \subset \spt_\mu(E)$, it follows that $E_1 \subset \spt_\mu(E) \setminus \partial_\mu E = \Int_\mu(E)$. Therefore $\Int_\mu(E) = E_1$. Exchanging the role of $E$ and $E^c$, we get that $\Int_\mu(E^c) = E_0$ and hence $\spt_\mu(E) = G \setminus \Int_\mu(E^c) = G \setminus  E_0$, which concludes the proof of \eqref{e:lem-interior-and-support}. To prove \eqref{e:lem-boundaries}, note that the first equality is an immediate consequence of \eqref{e:lem-interior-and-support}, namely,
\begin{equation*}
\partial_\mu E = \spt_\mu(E) \setminus \Int_\mu(E) = G\setminus (E_0 \cup E_1) = \partial^* E.
\end{equation*}
Next, recall that 
\begin{equation*}
\partial \Int_\mu(E) \cup \partial \Int_\mu(E^c) \subset  \partial_\mu E.
\end{equation*}
On the other hand, \eqref{e:lem-interior-and-support} and \eqref{e:lebesgue-representative} imply that $\mu(E \triangle \Int_\mu(E)) = \mu(E\triangle E_1) = 0$ and $\mu(E^c \triangle \Int_\mu(E^c)) =\mu(E^c\triangle E_0) = 0$ and we get from the definition of $\partial_\mu E$ that 
\begin{equation*}
\partial_\mu E \subset \partial \Int_\mu(E) \cap \partial \Int_\mu(E^c).
\end{equation*}
It follows that $\partial_\mu E = \partial \Int_\mu(E) = \partial \Int_\mu(E^c) = \partial (G\setminus \Int_\mu(E^c)) = \partial \spt_\mu(E)$, which concludes the proof of \eqref{e:lem-boundaries}. To prove \eqref{e:lem-regular-open-closed}, we use the fact that $\Int_\mu(E)$ is open together with \eqref{e:lem-boundaries} to get that 
\begin{equation*}
\overline{\Int_\mu(E)} = \Int_\mu(E) \cup \partial \Int_\mu(E) = \Int_\mu(E) \cup \partial_\mu E = \spt_\mu(E).
\end{equation*}
Exchanging the role of $E$ and $E^c$, we get 
\begin{equation*}
\overline{\Int_\mu(E^c)} =\spt_\mu(E^c) = G \setminus \Int_\mu (E).
\end{equation*}
Since $\spt_\mu(E) = G\setminus \Int_\mu(E^c)$, it follows that
\begin{equation*}
\Int(\spt_\mu(E)) = G \setminus \overline{\Int_\mu(E^c)} = \Int_\mu (E),
\end{equation*}
which concludes the proof.
\end{proof}

Theorem~\ref{thm:minimizer-for-perimeter} follows from Corollary~\ref{cor:main} and Lemma~\ref{lem:consequence-of-corollary}, recalling that if $E$ has locally finite perimeter then $\spt \Per_\bfX(E,\cdot) = \partial_\mu E$. The first part of Corollary~\ref{cor:minimizer-for-perimeter} is an immediate consequence of \eqref{e:interior-and-support} and \eqref{e:lebesgue-representative} while the second part follows from \eqref{e:boundaries} together with the local finiteness of the restriction of $\calH^{Q-1}_d$ to $\partial^* E$ for any set $E$ with locally finite perimeter and any left-invariant homogeneous distance $d$.

\begin{remark}  \label{rmk:AR-CB} Arguing as in \cite[Section~5]{MR2000099}, one can also deduce from local minimality and Lemma~\ref{lem:small-density} that if $E$ is a local minimizer for the $\bfX$-perimeter then $\partial_\mu E$ is a codimension-1 Ahlfors-regular set and $\Int_\mu(E)$ satisfies the so-called condition~B.
\end{remark}

\section{Measure-theoretic interior and support of monotone sets} \label{sect:int-spt-precisely-monotone}

 In this section we prove Theorem~\ref{thm:int-spt-precisely-monotone}, i.e., we prove that under assumption~(H) the measure-theoretic interior and support of a monotone set are precisely monotone. This will be obtained as a rather immediate consequence of Lemma~\ref{lem:main} which is inspired by \cite[Proposition~5.8]{MR2729270}.

\smallskip

We fix a scalar product $\scal{\cdot,\cdot}$ on $\g_1$. Given a line $L\in \bbL_{\scal{\cdot,\cdot}}(G)$ and $x,y\in L$, we denote by $(x,y)$ the open connected subset of $L$ with extremities $x$ and $y$.

\begin{lemma} \label{lem:main} Assume that (H) holds. Let $E$ be monotone, $L \in \bbL_{\scal{\cdot,\cdot}}(G)$, and $x,y \in L$. Then the following hold:
\begin{itemize}
\item[(i)] If $x\in \spt_\mu (E)$ and $y\in \Int_\mu(E)$ then the open segment $(x,y)$ is contained in $\Int_\mu(E)$.
\item[(ii)] If $x\in \spt_\mu (E^c)$ and $y\in \Int_\mu(E)$ then the connected component of $L\setminus \{y\}$ not containing $x$ lies in $\Int_\mu(E)$.
\end{itemize}
The same statements hold with the role of $E$ and $E^c$ exchanged.
\end{lemma}

\begin{proof} Let $E$ be monotone. Let $X\in \g_1 \setminus \{0\}$, $x\in G$, $p\in \N \setminus\{0,1\}$, $t>p$. Set $z=x\cdot \exp(pX) = x\cdot \Gamma_p(X,\dots,X)$ and $y= x\cdot \exp(tX)$. Assume that $y \in \Int_\mu(E)$. Then, arguing as in \cite[Lemma~5.9]{MR2729270}, we get that there is an open neighborhood $U$ of $x$ in $G$ and an open neighborhood $V$ of $(X,\dots,X)$ in $(\g_1)^p$ such that for $\mu$-a.e.~$x'\in U \cap E$ we have $x' \cdot \Gamma_p(X_1,\dots,X_p) \in E$ for a.e.~$(X_1,\dots,X_p) \in V$. Here we equip $(\g_1)^p$ with its canonical Lebesgue measure of dimension $p\dim\g_1$. If $\Gamma_p$ is open at $(X,\dots,X)$ then there is an open neighborhood $\calV$ of $\Gamma_p(X,\dots,X)$ such that $\calV \subset \Gamma_p(V)$. If moreover $x\in \spt_\mu(E)$ and since the map $x' \mapsto x'^{-1} \cdot z$ is continuous, one can find $x'\in U \cap E$ close enough to $x$ so that $z \in x' \cdot \calV$ and so that there is a full measure subset $V'$ of $V$ such that $x'\cdot \Gamma_p(V') \subset E$. Then we have 
\begin{equation*}
\mu (E^c \cap x' \cdot \calV) \leq \mu(x'\cdot (\Gamma_p(V) \setminus \Gamma_p(V'))) = \mu (\Gamma_p(V) \setminus \Gamma_p(V')).
\end{equation*}
Since $\Gamma_p$ is a smooth map, we know from Sard's Theorem that $\Gamma_p^{-1}(\{u\})$ is a $m$-dimensional submanifold of $(\g_1)^p$ for $\mu$-a.e.~$u \in \Gamma_p(V)$ where $m=p\dim \g_1 -\dim \g$. This implies in particular that $\calH^m(\Gamma_p^{-1}(\{u\})) >0$ where $\calH^m$ denotes the $m$-dimensional Euclidean Hausdorff measure on $(\g_1)^p$. On the other hand, the coarea formula asserts that for $\mu$-a.e.~$u \in G$ we have $\calH^m(\Gamma_p^{-1}(\{u\}) \cap V\setminus V')=0$. It follows that $\calH^m(\Gamma_p^{-1}(\{u\}) \cap V')>0$ and therefore $\Gamma_p^{-1}(\{u\}) \cap V' \not= \emptyset$ for $\mu$-a.e.~$u \in \Gamma_p(V)$, i.e., $\mu (\Gamma_p(V) \setminus \Gamma_p(V'))$. Therefore $x' \cdot \calV$ is an open neighborhood of $z$ such that $\mu (E^c \cap x' \cdot \calV) =0$, i.e., $z\in\Int_\mu(E)$.

Let us now prove (i). Let $x\in \spt_\mu(E)$ and $y\in \Int_\mu(E)$ be such that $x^{-1}\cdot y \in \exp(\g_1 \setminus\{0\})$. Let $z\in (x,y)$ and $Z\in \g_1 \setminus\{0\}$ be such that $z=x\cdot \exp Z$. By assumption~(H) there is an integer $p\geq 2$ such that the map $\Gamma_p$ is open at $(Z,\dots,Z) \in (\g_1)^p$. Since dilations are homeomorphisms and group homomorphisms, $\Gamma_p$ is open at $(\lambda Z,\dots,\lambda Z) \in (\g_1)^p$ for any $\lambda>0$. In particular $\Gamma_p$ is open at $(X,\dots,X)\in (\g_1)^p$ where $X = Z/p$ so that $z= x\cdot \Gamma_p(X,\dots,X)$. Then it follows from the previous observation that $z\in\Int_\mu(E)$, which concludes the proof of (i).

Since $E^c$ is monotone, statement~(i) also holds with the role of $E$ and $E^c$ exchanged. Since $\Int_\mu(E) \cap \Int_\mu(E^c) = \emptyset$, this implies that if $x \in \spt_\mu(E^c)$ and $y\in \Int_\mu(E)$ then the connected component of $L\setminus \{y\}$ not containing $x$ lies in $G \setminus \Int_\mu(E^c) = \spt_\mu(E)$. To conclude the proof of~(ii) we use statement~(i) for $E$ to infer that for any $z$ that lies in the connected component of $L\setminus \{y\}$ not containing $x$ we have $(y,z) \subset \Int_\mu(E)$. 
\end{proof}

\begin{proof}[Proof of Theorem~\ref{thm:int-spt-precisely-monotone}]
Let $E$ be monotone and $L\in \bbL_{\scal{\cdot,\cdot}}(G)$. Lemma~\ref{lem:main}(i) implies  that $\Int_\mu(E) \cap L$ is connected while Lemma~\ref{lem:main}(ii) implies that $\spt_\mu(E^c) \cap L = L\setminus \Int_\mu(E)$ is connected. Therefore $\Int_\mu(E)$ is precisely monotone. Exchanging the role of $E$ and $E^c$, we get that $\Int_\mu(E^c) = G\setminus \spt_\mu(E)$ is precisely monotone. Therefore $\spt_\mu(E)$ is precisely monotone.
\end{proof}

\begin{remark} \label{rmk:classification}
Theorem~\ref{thm:monotone-minimizer-for-perimeter} and Corollary~\ref{cor:minimizer-for-perimeter} combined with \cite[Theorem~4.3 and Proposition~5.8]{MR2729270} gives a rather short, and simplier, proof of the classification up to $\mu$-null sets of monotone subsets of the first Heisenberg group $\bbH$ \cite[Theorem~5.1]{MR2729270}. Indeed, let $E$ be a monotone subset of $\bbH$.  Then \cite[Proposition~5.8]{MR2729270}, which is the analogue of Lemma~\ref{lem:main} in the Heisenberg setting, implies that $\Int_\mu(E)$ is precisely monotone and therefore can be described as in \cite[Theorem~4.3]{MR2729270}. By Theorem~\ref{thm:monotone-minimizer-for-perimeter} and Corollary~\ref{cor:minimizer-for-perimeter} we know that $\Int_\mu(E)$ is equivalent to $E$ and it thus follows that $E$ is equivalent to either $\emptyset$, or $\bbH$, or a half-space. More generally, when condition~(H) holds, for instance in any step-2 Carnot group, one can reduce the classification up to $\mu$-null sets of monotone sets to the description of precisely monotone sets. This latter task is in principle easier than the study of monotone sets, but it is however still rather delicate and is known so far in only few other settings beyond the Heisenberg case, see~\cite{https://doi.org/10.48550/arxiv.2106.13490}. 
\end{remark}

\bibliography{bibliography} 

\begin{thebibliography}{10}

\bibitem{MR3353698}
L.~Ambrosio, R.~Ghezzi, and V.~Magnani.
\newblock B{V} functions and sets of finite perimeter in sub-{R}iemannian
  manifolds.
\newblock {\em Ann. Inst. H. Poincar\'{e} C Anal. Non Lin\'{e}aire},
  32(3):489--517, 2015.

\bibitem{MR1823840}
Luigi Ambrosio.
\newblock Some fine properties of sets of finite perimeter in {A}hlfors regular
  metric measure spaces.
\newblock {\em Adv. Math.}, 159(1):51--67, 2001.

\bibitem{MR1857292}
Luigi Ambrosio, Nicola Fusco, and Diego Pallara.
\newblock {\em Functions of bounded variation and free discontinuity problems}.
\newblock Oxford Mathematical Monographs. The Clarendon Press, Oxford
  University Press, New York, 2000.

\bibitem{MR2333095}
Vittorio Barone~Adesi, Francesco Serra~Cassano, and Davide Vittone.
\newblock The {B}ernstein problem for intrinsic graphs in {H}eisenberg groups
  and calibrations.
\newblock {\em Calc. Var. Partial Differential Equations}, 30(1):17--49, 2007.

\bibitem{MR3078944}
Costante Bellettini and Enrico Le~Donne.
\newblock Regularity of sets with constant horizontal normal in the {E}ngel
  group.
\newblock {\em Comm. Anal. Geom.}, 21(3):469--507, 2013.

\bibitem{MR4228617}
Costante Bellettini and Enrico Le~Donne.
\newblock Sets with constant normal in {C}arnot groups: properties and
  examples.
\newblock {\em Comment. Math. Helv.}, 96(1):149--198, 2021.

\bibitem{MR2729270}
Jeff Cheeger and Bruce Kleiner.
\newblock Metric differentiation, monotonicity and maps to {$L^1$}.
\newblock {\em Invent. Math.}, 182(2):335--370, 2010.

\bibitem{MR2892612}
Jeff Cheeger, Bruce Kleiner, and Assaf Naor.
\newblock Compression bounds for {L}ipschitz maps from the {H}eisenberg group
  to {$L_1$}.
\newblock {\em Acta Math.}, 207(2):291--373, 2011.

\bibitem{https://doi.org/10.48550/arxiv.2004.08129}
Enrico~Le Donne, Daniele Morbidelli, and Séverine Rigot.
\newblock Horizontally affine functions on step-2 carnot algebras.
\newblock \url{https://arxiv.org/abs/2004.08129}, 2020.

\bibitem{MR4127898}
Katrin F\"{a}ssler, Tuomas Orponen, and S\'{e}verine Rigot.
\newblock Semmes surfaces and intrinsic {L}ipschitz graphs in the {H}eisenberg
  group.
\newblock {\em Trans. Amer. Math. Soc.}, 373(8):5957--5996, 2020.

\bibitem{MR1871966}
Bruno Franchi, Raul Serapioni, and Francesco Serra~Cassano.
\newblock Rectifiability and perimeter in the {H}eisenberg group.
\newblock {\em Math. Ann.}, 321(3):479--531, 2001.

\bibitem{MR1984849}
Bruno Franchi, Raul Serapioni, and Francesco Serra~Cassano.
\newblock On the structure of finite perimeter sets in step 2 {C}arnot groups.
\newblock {\em J. Geom. Anal.}, 13(3):421--466, 2003.

\bibitem{MR3406514}
Matteo Galli and Manuel Ritor\'{e}.
\newblock Area-stationary and stable surfaces of class {$C^1$} in the
  sub-{R}iemannian {H}eisenberg group {$\Bbb{H}^1$}.
\newblock {\em Adv. Math.}, 285:737--765, 2015.

\bibitem{MR1404326}
Nicola Garofalo and Duy-Minh Nhieu.
\newblock Isoperimetric and {S}obolev inequalities for {C}arnot -
  {C}arath\'{e}odory spaces and the existence of minimal surfaces.
\newblock {\em Comm. Pure Appl. Math.}, 49(10):1081--1144, 1996.

\bibitem{MR3742567}
Enrico Le~Donne.
\newblock A primer on {C}arnot groups: homogenous groups,
  {C}arnot-{C}arath\'{e}odory spaces, and regularity of their isometries.
\newblock {\em Anal. Geom. Metr. Spaces}, 5(1):116--137, 2017.

\bibitem{MR3943324}
Enrico Le~Donne and S\'{e}verine Rigot.
\newblock Besicovitch covering property on graded groups and applications to
  measure differentiation.
\newblock {\em J. Reine Angew. Math.}, 750:241--297, 2019.

\bibitem{MR2000099}
Gian~Paolo Leonardi and S\'{e}verine Rigot.
\newblock Isoperimetric sets on {C}arnot groups.
\newblock {\em Houston J. Math.}, 29(3):609--637, 2003.

\bibitem{MR3330910}
Marco Marchi.
\newblock Regularity of sets with constant intrinsic normal in a class of
  {C}arnot groups.
\newblock {\em Ann. Inst. Fourier (Grenoble)}, 64(2):429--455, 2014.

\bibitem{MR1333890}
Pertti Mattila.
\newblock {\em Geometry of sets and measures in {E}uclidean spaces}, volume~44
  of {\em Cambridge Studies in Advanced Mathematics}.
\newblock Cambridge University Press, Cambridge, 1995.
\newblock Fractals and rectifiability.

\bibitem{MR4208090}
Annamaria Montanari and Daniele Morbidelli.
\newblock Multiexponential maps in {C}arnot groups with applications to
  convexity and differentiability.
\newblock {\em Ann. Mat. Pura Appl. (4)}, 200(1):253--272, 2021.

\bibitem{MR2165404}
Francescopaolo Montefalcone.
\newblock Some relations among volume, intrinsic perimeter and one-dimensional
  restrictions of {BV} functions in {C}arnot groups.
\newblock {\em Ann. Sc. Norm. Super. Pisa Cl. Sci. (5)}, 4(1):79--128, 2005.

\bibitem{MR4119259}
Daniele Morbidelli.
\newblock On the inner cone property for convex sets in two-step {C}arnot
  groups, with applications to monotone sets.
\newblock {\em Publ. Mat.}, 64(2):391--421, 2020.

\bibitem{https://doi.org/10.48550/arxiv.2106.13490}
Daniele Morbidelli and Séverine Rigot.
\newblock Precisely monotone sets in step-2 rank-3 {C}arnot algebras.
\newblock \url{https://arxiv.org/abs/2106.13490}, 2021.

\bibitem{MR3815462}
Assaf Naor and Robert Young.
\newblock Vertical perimeter versus horizontal perimeter.
\newblock {\em Ann. of Math. (2)}, 188(1):171--279, 2018.

\bibitem{https://doi.org/10.48550/arxiv.2004.12522}
Assaf Naor and Robert Young.
\newblock Foliated corona decompositions.
\newblock \url{https://arxiv.org/abs/2004.12522}, 2020.

\bibitem{MR3984100}
Sebastiano Nicolussi and Francesco Serra~Cassano.
\newblock The {B}ernstein problem for {L}ipschitz intrinsic graphs in the
  {H}eisenberg group.
\newblock {\em Calc. Var. Partial Differential Equations}, 58(4):Paper No. 141,
  28, 2019.

\bibitem{MR3587666}
Francesco Serra~Cassano.
\newblock Some topics of geometric measure theory in {C}arnot groups.
\newblock In {\em Geometry, analysis and dynamics on sub-{R}iemannian
  manifolds. {V}ol. 1}, EMS Ser. Lect. Math., pages 1--121. Eur. Math. Soc.,
  Z\"{u}rich, 2016.

\bibitem{MR4069613}
Francesco Serra~Cassano and Mattia Vedovato.
\newblock The {B}ernstein problem in {H}eisenberg groups.
\newblock {\em Matematiche (Catania)}, 75(1):377--403, 2020.

\bibitem{https://doi.org/10.48550/arxiv.2105.08890}
Robert Young.
\newblock Area-minimizing ruled graphs and the {B}ernstein problem in the
  {H}eisenberg group.
\newblock \url{https://arxiv.org/abs/2105.08890}, 2021.

\end{thebibliography}
\bibliographystyle{plain}

\end{document}